\documentclass[reqno]{amsart}
\usepackage{amssymb}
\usepackage{ifpdf}
\ifpdf
 \usepackage[hyperindex,pagebackref]{hyperref}
\else
 \expandafter\ifx\csname dvipdfm\endcsname\relax
 \usepackage[hypertex,hyperindex,pagebackref]{hyperref}
 \else
 \usepackage[dvipdfm,hyperindex,pagebackref]{hyperref}
 \fi
\fi
\allowdisplaybreaks[4]
\theoremstyle{plain}
\newtheorem{thm}{Theorem}

\theoremstyle{remark}
\newtheorem{rem}{Remark}

\begin{document}

\title[Properties of three functions relating to exponential function]
{Properties of three functions relating to the exponential function and the existence of partitions of unity}

\author[F. Qi]{Feng Qi}
\address{Department of Mathematics, School of Science, Tianjin Polytechnic University, Tianjin City, 300387, China; School of Mathematics and Informatics, Henan Polytechnic University, Jiaozuo City, Henan Province, 454010, China}
\email{\href{mailto: F. Qi <qifeng618@gmail.com>}{qifeng618@gmail.com}, \href{mailto: F. Qi <qifeng618@hotmail.com>}{qifeng618@hotmail.com}, \href{mailto: F. Qi <qifeng618@qq.com>}{qifeng618@qq.com}}
\urladdr{\url{http://qifeng618.wordpress.com}}



\begin{abstract}
In the paper, the author studies properties of three functions relating to the exponential function and the existence of partitions of unity, including accurate and explicit computation of their derivatives, analyticity, complete monotonicity, logarithmically complete monotonicity, absolute monotonicity, and the like.
\end{abstract}

\thanks{This paper was typeset using \AmS-\LaTeX}

\maketitle

\section{Introduction}

Throughout this paper, we denote the set of all positive integers by $\mathbb{N}$, and the set of all real numbers by $\mathbb{R}$.
\par
In the theory of differentiable manifolds, the function
\begin{equation}\label{part-unity-function}
f(t)=
\begin{cases}
e^{-1/t}, & t>0\\
0, & t\le0
\end{cases}
\end{equation}
plays an indispensable role in the proof of the existence of partitions of unity. See, for example, \cite[p.~10, Lemma~1.10]{Warner}.
\par
Similarly, we define
\begin{equation}
g(t)=e^{-1/t}, \quad t\in\mathbb{R}\setminus\{0\}
\end{equation}
and
\begin{equation}
h(t)=e^{1/t},\quad t\in\mathbb{R}\setminus\{0\}.
\end{equation}
It is clear that
\begin{equation}\label{g(-t)=h(t)}
g(-t)=h(t)
\end{equation}
for $t\ne0$ and $f(t)=g(t)$ for $t>0$.
\par
The task of this paper is to study properties of these three functions, including accurate and explicit computation of their derivatives, analyticity, complete monotonicity, absolute monotonicity, logarithmically complete monotonicity, and the like.

\section{Properties}

In this section, we will study properties of the functions $f(t)$, $g(t)$, and $h(t)$.

\subsection{Derivatives}
We first accurately compute derivatives of $i$-th order for $i\in\mathbb{N}$.

\begin{thm}\label{exp-frac1x-n-deriv-thm}
For $i\in\mathbb{N}$ and $t\ne0$, we have
\begin{equation}\label{exp-frac1x-expans}
h^{(i)}(t)=(-1)^ie^{1/t}\frac1{t^{2i}}\sum_{k=0}^{i-1}a_{i,k}t^{k},
\end{equation}
where
\begin{equation}\label{a-i-k-dfn}
a_{i,k}=\binom{i}{k}\binom{i-1}{k}{k!}
\end{equation}
for all $0\le k\le i-1$.
\end{thm}

\begin{proof}
We prove this theorem by induction on $i\in\mathbb{N}$.
\par
From
\begin{equation*}
h'(t)=-\frac1{t^2}e^{1/t},
\end{equation*}
it follows that
\begin{equation}\label{a-0-1=1}
a_{1,0}=1.
\end{equation}
This means that the equality~\eqref{exp-frac1x-expans} is valid for $i=1$.
\par
Assume the equality~\eqref{exp-frac1x-expans} is valid for some $i>2$.
\par
By virtue of the inductive hypothesis, a direct differentiation gives
\begin{align*}
h^{(i+1)}(t)&=\bigl[h^{(i)}(t)\bigr]'\\
&=\Biggl[(-1)^ie^{1/t}\frac1{t^{2i}}\sum_{k=0}^{i-1}a_{i,k}t^{k}\Biggr]'\\
&=(-1)^{i+1}e^{1/t}\frac1{t^{2(i+1)}}\Biggl(\sum_{k=0}^{i-1}a_{i,k}t^{k}
+2i\sum_{k=0}^{i-1}a_{i,k}t^{k+1} -\sum_{k=1}^{i-1}ka_{i,k}t^{k+1}\Biggr)\\
&=(-1)^{i+1}e^{1/t}\frac1{t^{2(i+1)}}\Biggl[\sum_{k=0}^{i-1}a_{i,k}t^{k}
+2i\sum_{k=1}^{i}a_{i,k-1}t^{k} -\sum_{k=2}^{i}(k-1)a_{i,k-1}t^{k}\Biggr]\\
&=(-1)^{i+1}e^{1/t}\frac1{t^{2(i+1)}}\Biggl\{a_{i,0} +\sum_{k=1}^{i-1}[a_{i,k}+(2i-k+1)a_{i,k-1}]t^{k} +(i+1)a_{i,i-1}t^i\Biggr\}\\
&=(-1)^{i+1}e^{1/t}\frac1{t^{2(i+1)}}\Biggl(a_{i+1,0} +\sum_{k=1}^{i-1}a_{i+1,k}t^{k} +a_{i+1,k}t^i\Biggr)\\
&=(-1)^{i+1}e^{1/t}\frac1{t^{2(i+1)}}\sum_{k=0}^{i}a_{i+1,k}t^{k}.
\end{align*}
The proof of Theorem~\ref{exp-frac1x-n-deriv-thm} is completed.
\end{proof}

\begin{thm}\label{g(t)-derivative-thm}
For $i\in\mathbb{N}$ and $t\ne0$,
\begin{equation}\label{g(t)-derivative}
g^{(i)}(t)=\frac1{e^{1/t}t^{2i}} \sum_{k=0}^{i-1}(-1)^ka_{i,k}{t^{k}},
\end{equation}
where $a_{i,k}$ is determined by~\eqref{a-i-k-dfn}.
\end{thm}

\begin{proof}
This follows readily from combination of
\begin{equation*}
g^{(i)}(t)=[h(-t)]^{(i)}=(-1)^ih^{(i)}(-t)
\end{equation*}
with Theorem~\ref{exp-frac1x-n-deriv-thm}.
\end{proof}

\begin{thm}\label{f(t)-derivative-thm}
For $i\in\mathbb{N}$,
\begin{equation}\label{f(t)-derivative}
f^{(i)}(t)=
\begin{cases}\displaystyle
\frac1{e^{1/t}t^{2i}} \sum_{k=0}^{i-1}(-1)^ka_{i,k}{t^{k}}, & t>0,\\
0, & t\le0,
\end{cases}
\end{equation}
where $a_{i,k}$ is determined by~\eqref{a-i-k-dfn}.
\end{thm}

\begin{proof}
This follows from the definition of $f(t)$ by~\eqref{part-unity-function} and Theorem~\ref{g(t)-derivative-thm}.
\end{proof}

\subsection{Analyticity}
A real function $q(t)$ is said to be analytic at a point $t_0$ if it possesses derivatives of all orders and agrees with its Taylor series in a neighborhood of $t_0$.

\begin{thm}
The function $f(t)$ defined by~\eqref{part-unity-function} is infinitely differentiable on $\mathbb{R}$ and
\begin{equation}\label{f-der-at0=0}
f^{(i)}(0)=0
\end{equation}
for all $i\in\mathbb{N}\cup\{0\}$, but it is not analytic at $t=0$.
\end{thm}

\begin{proof}
By Theorem~\ref{f(t)-derivative-thm}, it is obvious that, to prove the infinite differentiability of $f(t)$, it suffices to show the infinite differentiability of $f(t)$ at $t=0$. For this, we just need to compute the limit
\begin{equation}
f^{(i)}(0)=\lim_{t\to0^+}f^{(i)}(t)=0
\end{equation}
for all $i\ge0$, which can be calculated directly from~\eqref{f(t)-derivative}.
Consequently, the function $f(t)$ is infinitely differentiable on $(0,\infty)$.
\par
Suppose that the function $f(t)$ is analytic at $t=0$, then it should have a Taylor series
\begin{equation}\label{f(t)-exp-taylor}
f(t)=\sum_{\ell=0}^\infty\frac{f^{(\ell)}(0)}{\ell!}t^\ell
\end{equation}
for $t\in(-\delta,\delta)$, where $\delta>0$. Using~\eqref{f-der-at0=0}, the right hand side of~\eqref{f(t)-exp-taylor} becomes $0$ for all $t\in(-\delta,\delta)$. This contracts with the positivity of $f(t)$ on $(0,\delta)\subset(0,\infty)$. As a result, the function $f(t)$ is not analytic at $t=0$. The proof is complete.
\end{proof}

\begin{thm}
The point $t=0$ is a jump-infinite discontinuous point of the functions $g(t)$ and $h(t)$ on $\mathbb{R}$.
\end{thm}

\begin{proof}
This follows from
\begin{equation*}
\lim_{t\to0^-}h(t)=\lim_{t\to0^-}e^{1/t}=0=\lim_{t\to0^+}e^{-1/t}=\lim_{t\to0^+}g(t)
\end{equation*}
and
\begin{equation*}
\lim_{t\to0^+}h(t)=\lim_{t\to0^+}e^{1/t}=\infty=\lim_{t\to0^-}e^{-1/t}=\lim_{t\to0^-}g(t).
\end{equation*}
The proof is thus complete.
\end{proof}

\subsection{Complete monotonicity}
A function $q(x)$ is said to be completely monotonic on an interval $I$ if $q(x)$ has derivatives of all
orders on $I$ and
\begin{equation}\label{CM-dfn}
(-1)^{n}q^{(n)}(x)\ge0
\end{equation}
for $x\in I$ and $n\ge0$. See~\cite[Chapter~IV]{widder}.

\begin{thm}\label{h(t)-CM-thm}
The function $h(t)$ is completely monotonic on $(0,\infty)$.
\end{thm}

\begin{proof}
By Theorem~\ref{exp-frac1x-n-deriv-thm}, it follows that
\begin{equation*}
(-1)^ih^{(i)}(t)=e^{1/t}\frac1{t^{2i}}\sum_{k=0}^{i-1}a_{i,k}t^{k}>0
\end{equation*}
on $(0,\infty)$. So the function $h(t)$ is completely monotonic on $(0,\infty)$.
\end{proof}

\subsection{Logarithmically complete monotonicity}

A positive function $f(x)$ is said to be logarithmically completely monotonic on an interval $I\subseteq\mathbb{R}$ if it has derivatives of all orders on $I$ and its logarithm $\ln f(x)$ satisfies
\begin{equation}\label{lcm-dfn}
(-1)^k[\ln f(x)]^{(k)}\ge0
\end{equation}
for $k\in\mathbb{N}$ on $I$. A logarithmically completely monotonic function on $I$ must be completely monotonic on $I$, but not conversely. See~\cite{CBerg, absolute-mon-simp.tex, compmon2}, the expository article~\cite{bounds-two-gammas.tex} and closely-related references therein.

\begin{thm}
The function $h(t)$ and $\frac1{g(t)}$ are logarithmically completely monotonic on $(0,\infty)$, and so are the functions $g(t)$ and $\frac1{h(t)}$ on $(-\infty,0)$.
\end{thm}

\begin{proof}
This may be deduced from $\ln h(t)=\frac1t$ and $\ln g(t)=-\frac1t$ and standard arguments.
\end{proof}

\subsection{Absolute monotonicity}
An infinitely differentiable function $q(t)$ defined on an interval $I\subseteq\mathbb{R}$ is said to be absolutely monotonic if
\begin{equation}\label{log-abs-funct}
q^{(i)}(t)\ge0
\end{equation}
holds on $I$ for all $i\in\mathbb{N}\cup\{0\}$. See~\cite{absolute-mon-simp.tex, widder}.

\begin{thm}\label{g(t)-AM-thm}
The function $g(t)$ is absolutely monotonic on $(-\infty,0)$.
\end{thm}

\begin{proof}
This comes from~\eqref{g(t)-derivative} in Theorem~\ref{g(t)-derivative-thm}.
\end{proof}

\subsection{Remarks}

\begin{rem}
To the best of my knowledge, the derivatives~\eqref{exp-frac1x-expans} were described but without explicit and accurate expressions in many books.
\end{rem}

\begin{rem}
By definitions, it is easy to see that if $h(t)$ is completely monotonic on an interval $I$ then $h(-t)$ is absolutely monotonic on $-I$, the symmetrical interval of $I$ with respect to $0$. This implies that Theorem~\ref{h(t)-CM-thm} and Theorem~\ref{g(t)-AM-thm} are equivalent to each other.
\end{rem}

\end{document}